\documentclass{article}
\usepackage{amsmath}
\usepackage{amssymb}
\usepackage{geometry}
\usepackage{amsfonts}
\usepackage[utf8]{inputenc}
\usepackage{amsthm}
\usepackage{todonotes}

\newcommand{\de}{\mathrm d }
\newcommand{\R}{\mathbb R}

\newcommand{\E}{\mathbb E}
\newcommand{\eps}{\varepsilon}

\newcommand{\uMAP}{{u_{\textsc{map}}}}
\newcommand{\arginf}{\operatorname{arginf}}
\renewcommand{\H}{\mathrm H}

\newtheorem{lemma}{Lemma}
\newtheorem{proposition}{Proposition}
\newtheorem{corollary}{Corollary}

\newtheorem{claim}{Claim}
\theoremstyle{remark}
\newtheorem{rem}{Remark}

\title{Laplace's method in Bayesian inverse problems with Gaussian priors}
\author{Philipp Wacker\\Department of Mathematics\\Friedrich-Alexander-Universität Erlangen-Nürnberg (FAU)
}

\begin{document}
\maketitle
\begin{abstract}
In a Bayesian inverse problem setting, the solution consists of a posterior measure obtained by combining prior belief (where we restrict ourselves to Gaussian priors), information about the forward operator, and noisy observational data. This posterior measure is most often given in terms of a density with respect to a reference measure in a high-dimensional (or infinite-dimensional) Banach space. Although Monte Carlo sampling methods provide a way of querying the posterior, the necessity of evaluating the forward operator many times (which will often be a costly PDE solver) prohibits this in practice. For this reason, many practitioners choose a suitable Gaussian approximation of the posterior measure, by a procedure called Laplace's method. Once generated, this Gaussian measure is a easy to sample from and properties like moments are immediately acquired. This paper derives Laplace's approximation of the posterior measure attributed to the inverse problem explicitly by a second-order approximation of the data-misfit functional, specifically and rigorously in the infinite-dimensional setting. By use of a reverse Cauchy-Schwarz inequality we are able to explicitly bound the Hellinger distance between the posterior and its Laplace approximation.  
\end{abstract}
\section{Introduction}
We consider a inverse problem 
\begin{equation}\label{eq:IP}
y = G(u) + \eta
\end{equation}
where $G:X\to Y$ is a (possibly) nonlinear mapping between Hilbert spaces $X, Y$ and $\eta\sim N(0, \Gamma)$ is additive noise. The challenge consists of inferring the value of $u$ from the noisy (and usually lower-dimensional) observation $y$. This is an ill-defined problem in general, so some sort of regularization is needed.

In a Bayesian approach (see \cite{stuart2010inverse}) under Gaussian assumptions, we assume that $u\sim \mu_0 = N(0,C_0)$, i.e. we have a Gaussian prior on the variable $u$. For simplicity we assume that the mean is $0$, but this assumption can be dropped with slight modifications. The prior acts as a regularization and makes the inverse problem well-defined: Standard theory yields the posterior measure $\mu$ on $u$ given an observation $y$ under mild assumptions on the forward operator $G$:
\begin{equation}
\frac{\de \mu}{\de \mu_0}(u) = \frac{\exp(-\Phi(u))}{\int \exp(-\Phi(u)) \de\mu_0(u)},
\end{equation}
where $\Phi(u) = \frac{\|y-G(u)\|_{\Gamma}^2}{2}$.
Especially in higher dimensions, the posterior is in practice often approximated by a suitable Gaussian, in order to make computation of moments (or frequentist confidence sets) feasible. For a beautiful example how this is done in practice, see \cite{alexanderian2016fast} in the setting of optimal experimental design of an infinite-dimensional inverse problem, or in a finite-dimensional context in \cite{long2013fast}. This procedure is called Laplace's method and is the focus of this work. We define the functional $I(u) = \Phi(u) + \frac{\|u\|_{C_0}^2}{2}$ which can also be thought of as a (Tikhonov-)regularized cost functional in the classical sense. The \textit{maximum a posteriori point} is
\begin{equation}\label{eq:defMAP}
\uMAP := \arginf_u I(u)
\end{equation}
and the Laplace approximation is defined as
\begin{equation}
\nu = N(\uMAP, HI(\uMAP)^{-1})
\end{equation}
This means that the Laplace approximation is the Gaussian measure centered at the MAP point, with covariance operator matching the ``local'' covariance structure of the posterior measure. In finite dimensions, its density is exactly the normalized exponential of the measure's negative lognormal local quadratic approximation. For a good explanation (and as a general recommendation for a truly enjoyable book) of Laplace's method in finite dimensions, see \cite{mackay2003information}. The paper \cite{boone2005assessment} gives a numerical analysis of the Laplace approximation as well as the so-called Bayesian information criterion; \cite{friston2007variational} and the references therein give an overview over the use of Laplace's method in the Machine Learning and imaging community. A treatise about approximation of measures by Gaussian measures can be found in \cite{pinski2015kullback}, although they employ the Kullback-Leibler divergence (or relative entropy) as a notion of distance between measures. \cite{iglesias2013evaluation} is an extensive study of various Gaussian approximation methods (including Laplace's method) in the context of reservoir modelling. Standard results about Laplace's method are recorded in \cite{wong2001asymptotic} and \cite{breitung2006asymptotic}, with newer results on Laplace's method in \cite{kolokoltsov2015laplace}. Gaussian approximations in a different context have been treated in \cite{archambeau2007gaussian}, \cite{sanz2016gaussian} in the case of diffusion processes and in \cite{lu2016gaussian} in the context of molecular dynamics. \cite{stuart2016posterior} presents an approach running ``anti-parallely'' to Laplace's method: Instead of approximating the posterior measure directly, they approximate the forward operator or the negative log-likelihood. They bound the Hellinger distance between the true posterior and the resulting approximation, which is also the method of attack in this paper. \cite{chen2015geometric} discusses possible algorithms for implementing Laplace's method and similar approximation procedures in Bayesian inference problems. The MAP point is here defined as \eqref{eq:defMAP}, i.e. a minimization point of the functional $I$. See \cite{helin2015maximum} for a characteristation of the MAP point in a broader setting, by dropping the Gaussianity assumption (in particular, they derive an expression for the MAP point in a linear inverse problem with Besov priors). Original accounts of probabilistic methods of tackling inverse problems can be found in e.g. \cite{franklin1970well} for a random processes view in a Hilbert space setting, \cite{fitzpatrick1991bayesian} for -- among other things -- an explanation of how classical regularization can be viewed as application of a Bayesian prior. The difficulty of constructing well-defined linear estimators is addressed in \cite{mandelbaum1984linear} and \cite{luschgy1996linear}. In \cite{neubauer2008convergence}, a convergence result of the Ky-Fan metric (which is another metric between measures) between posterior measure and the minimum-norm-least-squares point (which is related to the maximum likelihood point) is proven. Although inverse problems in practice will always be finite-dimensional, it has proven worthwhile to study inverse problems from an infinite-dimensional viewpoint. This stems from the fact that although many algorithms work nicely for a given finite data set but break in the limit of finer and finer spatial or temporal resolution. One notable example is total variation regularisation, which is found to be edge-preserving, although it loses this property in the infinite-dimensional limit, see \cite{lassas2004can}. Another intuitive point is that algorithms that work in the infinite-dimensional setting should be complexity-invariant to a refinement of spatial resolution (naive algorithms often have high polynomial complexity with respect to the data). For a discussion of this fact see \cite{cotter2013mcmc}. For more details on how to construct well-defined infinite-dimensional inverse problems, see \cite{lassas2009discretization} and \cite{lasanen2002discretizations}.

It can be shown that the Laplace approximation $\nu$ coincides with the posterior measure $\mu$ if the forward operator $G$ is linear and the prior measure was Gaussian in the first place. Heuristically, the approximation is bad when the posterior measure is multimodal or has different tail properties than a Gaussian.

We are interested in deriving concrete error bounds for the approximation quality $\mu \approx \nu$. The Hellinger distance between probability measures lends itself to this cause. Given two measures $\mu,\nu$ which are absolutely continuous w.r.t. another measure $\mu_0$, the Hellinger distance (which is independent of the choice of $\mu_0$) between $\mu$ and $\nu$ is
 \begin{equation}
 (d_H(\mu,\nu))^2 = \frac{1}{2}\cdot {\int\left(\sqrt\frac{\de \mu}{\de\mu_0}- \sqrt\frac{\de \nu}{\de\mu_0}\right)^2\de \mu_0} = 1 - \int\sqrt\frac{\de \mu}{\de\mu_0}\sqrt\frac{\de \nu}{\de\mu_0}\de\mu_0
 \end{equation}

 The main conclusion of this paper is recorded informally in the following claims; they are correctly stated and proven later.
 \begin{claim}\label{claim:main1}
 While the posterior measure is given via its density w.r.t. the prior by
 \begin{equation*}
\frac{\de \mu}{\de \mu_0}(u) = \frac{\exp(-\Phi(u))}{\int \exp(-\Phi(u)) \de\mu_0(u)},
\end{equation*}
Laplace's method yields a Gaussian approximation of the form $\nu = N(\uMAP, HI(\uMAP)^{-1})$ and its density w.r.t. the prior is
\begin{equation}\label{eq:expressionnu}
\frac{\de \nu}{\de \mu_0}(u) = \frac{\exp(-T\Phi(u))}{\int \exp(-T\Phi(u)) \de\mu_0(u)}.
\end{equation}
Here, $T\Phi(u) = \Phi(\uMAP) + D\Phi(\uMAP)(u-\uMAP) + \frac{1}{2}H\Phi(\uMAP)[u-\uMAP, u-\uMAP]$, the second order Taylor approximation of the data-misfit functional $\Phi$ at $\uMAP$. 
 \end{claim}
  \begin{claim} \label{lem:main2}
If there is $K\in(0,1)$ such that
\begin{equation*}
\left\|\exp\left(-\frac{\Phi(u)}{2}\right) - \exp\left(-\frac{T\Phi(u)}{2}\right)\right\|_{L^2(X, \mu_0)} \leq K\cdot \frac{\exp(-I(\uMAP))}{\sqrt[4]{\det(C_0^{1/2}\cdot H\Phi(\uMAP)\cdot C_0^{1/2} + \operatorname{Id})}},
\end{equation*}
then the Hellinger distance between the posterior and its Laplace's approximation can be bounded:
\begin{equation*}
d_H(\mu,\nu) \leq \frac{K}{\sqrt{1 + (1-K)^2}}
\end{equation*}
 \end{claim}
 While approximation results for Laplace's method have been achieved before (e.g. in the references made above), it is to the best of our knowledge that the claims as stated have not been formulated in an Inverse problem setting where Laplace's method is often used.
 
 The paper is organized as follows: First the more intuitive one-dimensional case is presented and representation \eqref{eq:expressionnu} is derived. The approach taken can not be generalized to the infinite dimensional case, as we use densities with respect to a Lebesgue measure. Then the infinite-dimensional equivalent is proven, where we show the equality \eqref{eq:expressionnu} directly by means of characteristic functions. The following section shows that the problem of bounding the Hellinger distance can be reduced to a reverse Cauchy-Schwarz inequality. After recording a few elementary results about reverse CS inequalities, we immediately obtain claim \ref{lem:main2} and also a more practical (but less tight) version of it.
 
 \begin{rem}
 A caveat regarding practicality: Although proposition \ref{prop:main2} and corollary \ref{cor:useful} yield expressions that seem to offer readily-checked conditions for explicit bounds on the Hellinger distance, a non-artificially constructed example where claim 2 easily yields a non-trivial bound on $d_H$ seems hard do obtain and the applicability of the second claim should thus not be overstated. This is explained in more detail in the section about application. Rather, the main point of this paper is the derivation and proof of characterization of expression \eqref{eq:expressionnu}.
 \end{rem}
\section{The Laplace approximation in one dimension}
We recall 
\begin{equation*}
\frac{\de \mu}{\de \mu_0}(u) = \frac{\exp(-\Phi(u))}{\int \exp(-\Phi(u)) \de\mu_0(u)},
\end{equation*}
and with for the Lebesgue measure $\lambda$ in one dimension, the Laplace approximation of $\mu$, which is a Gaussian centered at the maximum a posteriori point, with variance equal to the inverse  of the second derivative of $I$ at this point. In finite dimensions, any measure is most easily recorded by its density w.r.t. $\lambda$.
\begin{equation}
\frac{\de \nu}{\de \lambda}(u) = \sqrt\frac{I''(\uMAP)}{2\pi} \exp\left( -\frac{I''(\uMAP)}{2}\cdot (u-\uMAP)^2\right),
\end{equation}
hence
\begin{equation}
\frac{\de\nu}{\de\mu_0}(u) = \frac{\frac{\de\nu}{\de\lambda}(u)}{\frac{\de\mu_0}{\de\lambda}(u)} = \sqrt{I''(\uMAP)}\cdot\sigma\cdot  \exp\left( -\frac{I''(\uMAP)}{2}\cdot (u-\uMAP)^2 + \frac{u^2}{2\sigma^2}\right).
\end{equation}
As we mentioned, second-order Taylor approximations will play a leading role, hence we define 
\begin{equation}
R(u) = I(u) - I(\uMAP) - \frac{I''(\uMAP)}{2}(u-\uMAP)^2,
\end{equation}
the error term of the second order Taylor approximation of $I$ in $\uMAP$. Note that the first-order term vanishes because of $\uMAP$ being a minimum of $I$, which we assume to be $C^2$. Interestingly, $R(u)$ is also an error term for the second order approximation of $\Phi$, albeit in a slightly different way:
\begin{align*}
R(u) &= \Phi(u) + \frac{u^2}{2\sigma^2} - \Phi(\uMAP) - \frac{\uMAP^2}{2\sigma^2}-\frac{\Phi''(\uMAP)}{2}(u-\uMAP)^2 - \frac{1}{\sigma^2}(u-\uMAP)^2\\
&= \frac{\uMAP\cdot (u-\uMAP)}{\sigma^2} + \Phi'(\uMAP)(u-\uMAP) \\
&+\left[ \Phi(u) - \Phi(\uMAP) - \Phi'(\uMAP)(u-\uMAP) - \frac{\Phi''(\uMAP)}{2}(u-\uMAP)^2\right]\\
&= \Phi(u) - T^{(2)}_\uMAP\Phi(u).
\end{align*}
Note that the terms in the second line vanish because of $0 = I'(\uMAP) = \Phi'(\uMAP) + \frac{\uMAP}{\sigma^2}$ and the second term is exactly the error between $\Phi(u)$ and its second order Taylor polynom developed in $\uMAP$. It holds that $I'(\uMAP) = 0$ but in general $\Phi'(\uMAP)\neq 0$.
With this definition of $R$, and especially
\[ -\frac{I''(\uMAP)}{2}\cdot (u-\uMAP)^2 + \frac{u^2}{2\sigma^2} = -\Phi(u) + I(\uMAP) + R(u),\]
we obtain for the density of $\nu$ w.r.t. $\mu_0$:
\begin{equation*}
\frac{\de\nu}{\de\mu_0}(u) = \sqrt{I''(\uMAP)}\cdot\sigma\cdot \exp(I(\uMAP)) \cdot \exp(-\Phi(u) + R(u)).
\end{equation*}
An easy calculation shows
\begin{equation}\label{eq:calcTPhi}
\begin{split}
\frac{\exp(-I(\uMAP))}{\sigma \sqrt{I''(\uMAP)}} &= \int \exp \left(\frac{u^2}{2\sigma^2-I(\uMAP)-\frac{I''(\uMAP)}{2}(u-\uMAP)^2}\right)\mu_0(\de u)\\
&= \int \exp\left(-\Phi(u)+R(u)\right)\mu_0(\de u) = \int \exp(-T^{(2)}_\uMAP\Phi(u))\de\mu_0(u).
\end{split}
\end{equation}
This also follows immediately from the normalization of $\de\nu/\de\mu_0$ and thus
\begin{equation}
\frac{\de\nu}{\de\mu_0}(u) = \frac{\exp(-\Phi(u)+R(u))}{\int \exp(-\Phi(u)+R(u))\mu_0(\de u)} = \frac{\exp(-T\Phi(u))}{\int \exp(-T\Phi(u))\mu_0(\de u)}.
\end{equation}

\section{The Laplace approximation in general Hilbert spaces}
In this section, we will show claim \ref{claim:main1} in the setting of a general (possibly infinite-dimensional) Hilbert space. In order to do this, we need a slight adaptation of a useful Gaussian integral calculation.

In Proposition 1.2.8 in \cite{da2002second}, the authors prove the following: Let $\mu_0 = N(0, Q)$ be a Gaussian measure on a real Hilbert space $H$. Assume that $M$ is a symmetric operator such that $\langle Q^{1/2}MQ^{1/2} u, u\rangle < \langle u,u\rangle$ for all $0\neq u\in H$. Then for $b\in H$
\begin{equation}\label{eq:DPZ}
\begin{split}
\int_H \exp\left(\frac{1}{2}\langle M u,u\rangle + \langle b, u\rangle\right) \de\mu_0(u) = \frac{\exp\left(\frac{1}{2}\left|(1-Q^{1/2}MQ^{1/2})^{-1/2} \cdot Q^{1/2}b\right|^2 \right)}{\sqrt{\det (1-Q^{1/2}MQ^{1/2})}}.
\end{split}
\end{equation}
We will need a generalization of this formula, which follows from analytical continuation of the (real) Hilbert space's inner product to its complex extension. Recall that this continuation will not be positively definite anymore: $\langle \lambda_1 a + b, \lambda_2 c + d\rangle = \lambda_1 \lambda_2 \langle a, c\rangle + \lambda_2 \langle b, c\rangle + \lambda_1\langle a, d\rangle + \langle b,d\rangle$, without any complex conjugations on any $\lambda_i$. The following lemma is stated without proof as it follows immediately from the bilinearity of the analytical continuation stated above.

\begin{lemma}
Let $\mu_0 = N(0, Q)$ be a Gaussian measure on a real Hilbert space $H$. Assume that $M$ is a symmetric operator such that $\langle Q^{1/2}MQ^{1/2} u, u\rangle < \langle u,u\rangle$ for all $0\neq u\in H$. Then, with $L:=Q^{1/2}(1-Q^{1/2}MQ^{1/2})^{-1}Q^{1/2}$ and for $b_1, b_2\in H$
\begin{equation}\label{eq:DPZ_adapted}
\begin{split}
\int_H &\exp\left(\frac{1}{2}\langle M u,u\rangle + \langle b_1+ib_2, u\rangle\right) \de\mu_0(u) \\
&= \frac{\exp\left(\frac{1}{2}\langle Lb_1,b_2\rangle + i\langle Lb_1, b_2\rangle - \frac{1}{2}\langle Lb_2, b_2\rangle  \right)}{\sqrt{\det (1-Q^{1/2}MQ^{1/2})}}.
\end{split}
\end{equation}
\end{lemma}

Now we can prove our main result: 
\begin{proposition}\label{prop:main1}
Consider the inverse problem \ref{eq:IP} with prior $\mu_0$ and posterior $\mu$ given by \[\frac{\de\mu}{\de\mu_0}(u) = \frac{\exp(-\Phi(u))}{\int\exp(-\Phi(u))\de\mu_0}.\] 
The functional $I(u) = \Phi(u) + \frac{1}{2}\|u\|_{C_0}^2$ is assumed to be $C^2$ in a neighborhood of $\uMAP = \arginf I(u)$.
Then the Laplace approximation of $\mu$ given by
\[ \nu = N(\uMAP, \H I(\uMAP)^{-1}) \]
is equivalently defined by
\[ \frac{\de\nu}{\de\mu_0}(u) = \frac{\exp(-T\Phi(u))}{\int\exp(-T\Phi(u))\de\mu_0},\] 
where $T\Phi(u) = \Phi(\uMAP) + D\Phi(\uMAP)(u-\uMAP) + \frac{1}{2}\H\Phi(\uMAP)[u-\uMAP, u-\uMAP]$ is the second order Taylor approximation of $\Phi$ generated in $\uMAP$

\end{proposition}

\begin{proof}
This is done by comparing the Fourier transform (or characteristic function) of both representations for $\nu$. We calculate

\begin{equation}\label{eq:normalization}
\int \exp(-T\Phi(u))\de\mu_0(u) = \frac{e^{-I(\uMAP)}}{\sqrt{\det C^{1/2}\H I(\uMAP)C^{1/2}}}
\end{equation}
and
\begin{equation}\label{eq:charFnc}
\int \exp(i\langle \lambda, u\rangle -T\Phi(u))\de\mu_0(u) = \frac{e^{-I(\uMAP)}\cdot \exp\left( i\langle \uMAP, \lambda\rangle - \frac{1}{2}\cdot \H I(\uMAP)^{-1}[\lambda, \lambda]\right)}{\sqrt{\det C^{1/2}\H I(\uMAP)C^{1/2}}} 
\end{equation}
We show \eqref{eq:charFnc}, as \eqref{eq:normalization} follows from setting $\lambda = 0$. We use 
\[-T\Phi(u) = R(u) - \Phi(u) = \frac{\|u\|_C^2}{2} - I(\uMAP) - \frac{1}{2}\H I(\uMAP)[u-\uMAP, u-\uMAP].\]
and write $J = \H I(\uMAP)$ and $v=\uMAP$ for brevity. Note that for a bilinear operator $K$ we will identify $K(w, z) = \langle K w, z\rangle$. Then \begin{align*}
\int \exp(i\langle \lambda, u\rangle &-T\Phi(u))\de\mu_0(u)\\
 &= \int \exp\left(i\langle \lambda, u \rangle  + \frac{\|u\|_C^2}{2} - I(v) - \frac{1}{2}J[u-v, u-v] \right)\de\mu_0(u)\\
&=e^{-I(v)}\cdot \int \exp\left( \langle i\lambda, u\rangle + \frac{1}{2}\langle C^{-1} u,u\rangle  -\frac{1}{2}\langle J(u-v),u-v\rangle \right)\de\mu_0(u)\\
&= e^{-I(v) - \frac{1}{2}\langle Jv,v\rangle}\cdot \int \exp\left(\langle Jv + i\lambda, u\rangle + \frac{1}{2} \langle (C^{-1}-J)u,u\rangle \right)\de\mu_0(u).
\intertext{This is formula \eqref{eq:DPZ_adapted} with $M = C^{-1}-J$ and $b_1 = Jv$, $b_2=\lambda$. In this case, $1 - C^{1/2}MC^{1/2} = 1 - C^{1/2}(C^{-1}-J)C^{1/2} = C^{1/2}JC^{1/2}$ and thus $(1 - C^{1/2}MC^{1/2})^{-1/2} = J^{-1/2}C^{-1/2}$. Continuing,}
&= e^{-I(v) - \frac{1}{2}\langle Jv,v\rangle}\cdot \frac{\exp\left(\frac{1}{2}|J^{1/2}v|^2 + i \langle J^{1/2}v,J^{-1/2}\lambda \rangle- \frac{1}{2}|J^{-1/2}\lambda|^2\right)}{\sqrt{\det C^{1/2}JC^{1/2}}}\\
&= e^{-I(v)}\cdot \frac{\exp\left(i \langle v,\lambda \rangle - \frac{1}{2}\langle J^{-1}\lambda, \lambda\rangle\right)}{\sqrt{\det C^{1/2}JC^{1/2}}}
\end{align*}
This proves equation \eqref{eq:charFnc} and it follows that the characteristic function of the measure $\tilde\nu$ defined by $\frac{\de\tilde\nu}{\de\mu_0} = 1/Z\cdot \exp(-T\Phi)$ fulfills
\begin{align*}
\hat{\tilde\nu}(\lambda) &= \frac{\int \exp(i\langle \lambda, u\rangle) \de\tilde\nu(u)}{Z} = \frac{\int \exp(i\langle \lambda, u\rangle)\cdot \exp(-T\Phi(u)) \de\mu_0(u)}{\int \exp(i\langle \lambda, u\rangle) \de\mu_0(u)}\\
&= \exp\left\{i\langle \uMAP, \lambda\rangle + \frac{1}{2}\langle \H I(\uMAP)^{-1}(u),u\rangle \right\},
\end{align*}
i.e. $\tilde \nu = N(\uMAP, \H I(\uMAP)^{-1})$ as claimed.
\end{proof}
As in the one-dimensional setting, we conclude
\[\frac{\de\mu}{\de\mu_0} = \frac{\exp(-\Phi(u))}{\int \exp(-\Phi(u))\de\mu_0(u)}\]
and
\[\frac{\de\nu}{\de\mu_0} = \frac{\exp(-T\Phi(u))}{\int \exp(-T\Phi(u))\de\mu_0(u)}.\]

With these expressions, we derive a bound on the Hellinger distance between  $\mu$ and $\nu$ in the next section.
\section{Hellinger distance}

There are many notions of metrics and semi-metrics between measures, notably total variation, Hellinger, Wasserstein, Prokhorov and Kullback-Leibler. A survey of these and more probability metrics, including a detailed exposition of their relations is \cite{gibbs2002choosing}. 

We choose the Hellinger distance, mainly because of its good analytic properties, its consistency with the total variation metric and because of the following lemma which allows us to bound the difference between expectations under the different measures in question:

\begin{lemma}[part of Lemma 7.14 in \cite{dashti2013bayesian}]
Let $\mu, \mu'$ be two probability measures which are absolutely continuous w.r.t. another measure $\nu$ on a Banach space $(X,\|\cdot \|_X)$. Assume that $f: X\to E$, where $(E,\|\cdot\|)$ has second moments with respect to both $\mu$ and $\mu'$. Then
\[ \|\E^\mu f - \E^{\mu'} f\| \leq 2 \sqrt{\E^\mu \|f\|^2 + \E^{\mu'}\|f\|^2}\cdot d_H(\mu,\mu').\]
\end{lemma}

Recall that

\[d_H(\mu,\nu)^2 = {1 - \int \sqrt{\frac{\de\mu}{\de\mu_0}(u)}\sqrt{\frac{\de\nu}{\de\mu_0}(u)}\de\mu_0(u)}.\]
Thus, with the results of the preceding sections,
\begin{equation}\label{eq:exprdH}
\begin{split}
d_H(\mu,\nu)^2 &= 1 - \frac{\int \exp(-\frac{\Phi(u) + T\Phi(u)}{2})\de\mu_0(u)}{\sqrt{\int \exp(-\Phi(u))\de\mu_0(u)}\sqrt{\int \exp(-T\Phi(u))\de\mu_0(u)}}\\
&= 1 - \frac{\left\langle \exp\left(-\frac{1}{2}\Phi \right), \exp\left(-\frac{1}{2}T\Phi \right)\right\rangle_{L^2(X,\mu_0)}}{\left\|\exp\left(-\frac{1}{2}\Phi \right)\right\|_{L^2(X,\mu_0)}\cdot \left\|\exp\left(-\frac{1}{2}T\Phi \right)\right\|_{L^2(X,\mu_0)}}
\end{split}
\end{equation}

\begin{rem}
From positivity of the exponential and the Cauchy-Schwarz-Bunyakowsky inequality it can easily be seen that $d_H(\mu,\nu)\in [0,1]$.
\end{rem}
We would like to bound the Hellinger distance, i.e. optimally we would like to prove something like $d_H(\mu,\nu) \leq \epsilon$ for some $\epsilon > 0$. This amounts to a reverse Cauchy-Schwarz inequality, or a statement of the kind
\begin{equation}\label{eq:revCSI}
\langle e^{-\Phi/2}, e^{-(T\Phi)/2} \rangle_{L^2(X,\mu_0)} > (1-\eps^2) \cdot  \|e^{-\Phi/2}\|_{L^2(X,\mu_0)} \cdot \|e^{-(T\Phi)/2}\|_{L^2(X,\mu_0)}.
\end{equation}
In the next section, we present a few elementary results about this kind of inequality.
\section{A reverse Cauchy-Schwarz inequality}\label{sec:reverse}
In this section, $H$ will always be a Hilbert space with inner product $\langle\cdot\rangle$ and norm $\|\cdot\|$. The study of of reverse Cauchy-Schwarz-Bunyakowsky inequalities is an active field of research by itself. We refer to \cite{dragomir2015reverses} for further reading, although we will need only a very basic form of a reverse CSB inequality.

%in a slightly simplified version with its proof as we will use it in the following:
%\todo[inline]{redundante Aussage }
%\begin{thm}[Corollary 4 in \cite{dragomir2015reverses}]\label{thm:revCS}
%Let $a,A>0$. If $x,y\in H$, with $H$ a Hilbert space over the reals, are such that 
%\[\langle Ay-x,x-ay\rangle \geq 0,\]
%then 
%\[ \|x\|\cdot \|y\| \leq \frac{A+a}{2\sqrt{aA}}\langle x, y\rangle.\]
%\end{thm}
%\begin{proof}
%By assumption,
%\begin{align*}
%(A+a)\langle x,y\rangle\geq \|x\|^2 + aA\|y\|^2
%\intertext{and thus}
%\frac{A+a}{\sqrt{aA}}\langle x,y\rangle\geq \frac{1}{\sqrt{aA}}\|x\|^2 + \sqrt{aA}\|y\|^2 \geq 2\|x\|\cdot \|y\|
%\end{align*}
%where the last inequality is Young's inequality. 
%\end{proof}
%The two following elementary lemmata will be used in order to utilize \ref{thm:revCS} better.
\begin{lemma}\label{lem:aux1}
Let $f,g\in H$ and $D>0$ with $\|f-g\|^2 \leq D\cdot (\|f\|^2 + \|g\|^2)$. Then
\begin{equation}\label{eq:reverseYoung}
\langle f, g \rangle \geq \frac{1-D}{2}\cdot (\|f\|^2+\|g\|^2)\geq (1-D)\cdot \|f\|\cdot \|g\|.
\end{equation} 
\end{lemma}
\begin{proof}
\begin{align*}
\langle f, g \rangle &- \frac{1-D}{2}\cdot \|f\|^2 - \frac{1-D}{2}\cdot \|g\|^2 \\
&= \frac{1}{2} \|f\|^2 + \frac{1}{2} \|g\|^2 - \frac{1}{2} \|f-g\|^2- \frac{1-D}{2}\cdot \|f\|^2 - \frac{1-D}{2}\cdot \|g\|^2\\
&= \frac{1}{2}\left[D\cdot (\|f\|^2+\|g\|^2)-\|f-g\|^2\right] \geq 0.
\end{align*}
The last inequality in \eqref{eq:reverseYoung} is just Young's inequality in $\R$.
\end{proof}
\begin{rem} Lemma \ref{lem:aux1} can be thought of as a reverse Young's inequality (which gives for all $f,g\in H$ that $\langle f, g\rangle \leq \frac{1}{2}\|f\|^2 + \frac{1}{2}\|f\|^2 $ and thus if the conditions on $f,g$ as in lemma \ref{lem:aux1} are fulfilled, we can bound
\[\frac{1-D}{2}\cdot (\|f\|^2+\|g\|^2) \leq \langle f, g\rangle \leq \frac{1}{2}\cdot (\|f\|^2+\|g\|^2). \] 
\end{rem}
\begin{lemma}\label{lem:aux2}
Let $f,g\in H$ and $K \in (0,1)$ with $\|f-g\| \leq K\cdot \|f\|$. Then
\[ \langle f,g\rangle \geq \left[\frac{1-K}{1 + (1-K)^2}\right]\cdot \left(\|f\|^2+\|g\|^2\right) \geq 2 \cdot \left[\frac{1-K}{1 + (1-K)^2}\right] \cdot \|f\|\cdot \|g\|.\]
\end{lemma}
\begin{proof}
\begin{align*}
\frac{\|f-g\|^2}{\|f\|^2+\|g\|^2}&\leq \frac{\|f-g\|^2}{\|f\|^2 + (\|f\|-\|f-g\|)^2}= \frac{\|f-g\|^2}{2\|f\|^2 - 2\|f\|\|f-g\| + \|f-g\|^2}\\
&\leq \frac{\|f-g\|^2}{2\|f\|^2 - 2K\|f\|^2 + \|f-g\|^2}= 1 - \frac{(2-2K)\|f\|^2}{(2-2K)\|f\|^2 + \|f-g\|^2}\\
&\leq 1 - \frac{2(1-K)}{1 + (1-K)^2}
\end{align*}
And using lemma \ref{lem:aux1} (with $D = 1 - \frac{2(1-K)}{1 + (1-K)^2}$), we obtain the result.
\end{proof}
\begin{rem}
This is another form of a reverse Young's inequality, but with a different prerequisite: If $\|f-g\|\leq K\cdot \|f\|$, we have
\begin{equation}
\frac{1}{2}\cdot \left(1 - \frac{K^2}{1 + (1-K)^2}\right) \cdot (\|f\|^2 + \|g\|^2) \leq \langle f,g\rangle \leq \frac{1}{2}\cdot (\|f\|^2+\|g\|^2)
\end{equation}
\end{rem}

%\begin{lemma}[alte Version] \label{lem:aux2_old}
%Let $f,g\in H$ and $K \in (0,1)$ with $\|f-g\| \leq K\cdot \|f\|$. Then
%\[ \langle f,g\rangle \geq \left[\frac{1}{2} - \frac{K^2}{(2-K)^2}\right]\cdot \left(\|f\|^2+\|g\|^2\right).\]
%\end{lemma}
%\begin{proof}
%From assumption it follows that 
%\begin{align*}
%\|f-g\|(1-K/2) &\leq K/2 \cdot (2\|f\| - \|f-g\|)\\
%&\leq K/2\cdot (\|f\| + \|g\|)
%\intertext{and hence with $(a+b)^2 \leq 2a^2 + 2b^2$}
%\|f-g\|^2 &\leq \frac{2K^2}{(2-K)^2}\cdot (\|f\|^2 + \|g\|^2).
%\end{align*}
%Lemma \ref{lem:aux1} with $D:=\frac{2K^2}{(2-K)^2}$ immediately yields the statement.
%\end{proof}
\section{Conditions for good approximation}
From lemma \ref{lem:aux2} and the expression for $\de_H(\mu, \nu)$ in \eqref{eq:exprdH} (the Hellinger distance between the posterior measure $\mu$ and its Laplace approximation $\nu$) we immediately obtain the following:
\begin{proposition}\label{prop:main2}
With the assumptions of proposition \ref{prop:main1}, if for some $K\in (0,1)$
\begin{equation}\label{eq:condmain2}
\begin{split}
\left\|\exp\left(-\frac{\Phi(u)}{2}\right) - \exp\left(-\frac{T\Phi(u)}{2}\right)\right\|_{L^2(X, \mu_0)} &\leq K\cdot \frac{\exp(-I(\uMAP))}{\sqrt[4]{\det(C_0^{1/2}\cdot \H I(\uMAP)\cdot C_0^{1/2})}}\\
&= K\cdot \frac{\exp(-I(\uMAP))}{\sqrt[4]{\det(\operatorname{Id} + C_0^{1/2}\cdot H\Phi (\uMAP)\cdot C_0^{1/2})}}
\end{split}
\end{equation}
then 
\begin{equation}
d_H(\mu,\nu) \leq \frac{K}{\sqrt{1 + (1-K)^2}}
\end{equation}
\end{proposition}
\begin{proof}
We use lemma \ref{lem:aux2}, where $H = L^2(X,\mu_0)$, $f=\exp(-T\Phi/2)$ and $g = \exp(-\Phi/2)$.  Note that $\|f\|_H^2 = \exp(-I(\uMAP))/\sqrt{\det(C_0^{1/2}\cdot \H I(\uMAP) \cdot C_0^{1/2})}$ due to \eqref{eq:normalization}. Then we can set $1-\eps^2 = 2  \left[\frac{1-K}{1 + (1-K)^2}\right]$, hence $\eps = \frac{K}{\sqrt{1 + (1-K)^2}}$ in equation \eqref{eq:revCSI}.

The equality in condition  \eqref{eq:condmain2} follows from $I(u) = \Phi(u) + \frac{1}{2}\|u\|_{C_0}$.

%$A + \frac{1}{A} = \left[\frac{1-K}{1 + (1-K)^2}\right]$, which is $A= \frac{1}{1-K}$ and hence corollary \ref{rem:easiercase} yields the result
%\[d_H(\mu,\nu) \leq \frac{\frac{1}{1-K}-1}{\sqrt{\frac{1}{(1-K)^2}+1}} = \frac{K}{\sqrt{1 + (1-K)^2}}.\]
\end{proof}

\begin{rem}
Note that 
\begin{equation}
\begin{split}C_0^{1/2}H\Phi(u)C_0^{1/2}[h_1,h_2] &= \langle DG(u)(C_0^{1/2}h_1), \Gamma^{-1}DG(u)(C_0^{1/2}h_2)\rangle \\
&- \langle HG(u)[C_0^{1/2}h_1, C_0^{1/2}h_2], \Gamma^{-1}(y-G(u))\rangle\end{split}\end{equation}
where $DG$ is the Frechet derivative of the forward operator of $G$ and $HG$ is its second Frechet derivative (the Hessian).
\end{rem}

The following corollary uses assumptions which can be checked more easily but is less strict.
\begin{corollary}\label{cor:useful}
Define 
\[ K^2:= \frac{\sqrt{\det(C_0^{1/2}\H I(\uMAP)C_0^{1/2})}}{\exp(-I(\uMAP))}\cdot \int \exp (-\min\{\Phi(u),T\Phi(u)\})\cdot \min\left\{\frac{|\Phi(u)-T\Phi(u)|^2}{4},1\right\}\mu_0(\de u). \]
Then we have
\[d_H(\mu,\nu) \leq \frac{K}{\sqrt{1 + (1-K)^2}}.\]
\end{corollary}
\begin{proof}
This is due to the elementary inequality $|e^{-x}-e^{-y}|\leq e^{-(x\wedge y)}\cdot (|x-y|\wedge 1)$ from which we obtain
\begin{align*}
\int |e^{-\Phi(u)/2}-e^{-T\Phi(u)/2}|^2\de\mu_0 &\leq \int e^{-(\Phi(u)\wedge T\Phi(u))} \cdot (|\Phi(u)-T\Phi(u)|^2/4 \wedge 1)\de\mu_0\\
&=K^2 \cdot \int e^{-T\Phi(u)}\de\mu_0.
\end{align*}
\end{proof}
\section{A caveat on application}

We check on a one-dimensional example whether the bounds on the Hellinger distance obtained are tight. We consider a one-dimensional example by setting $G:\R\to\R$, with $G(u) = \exp(u)$. This means that $\Phi(u) = \frac{|y-e^u|^2}{2}$.

We consider two cases: $y=-2$ and $y=2$. In the first case, the Laplace approximation is not very good ($d_H(\mu,\nu) \approx 0.3260$) whereas in the second case, there is only a slight misfit between posterior measure and its Laplace approximation ($d_H(\mu, \nu)\approx 0.0958 $).

\begin{figure}[hbtp]
\centering
\includegraphics[width=0.8\textwidth]{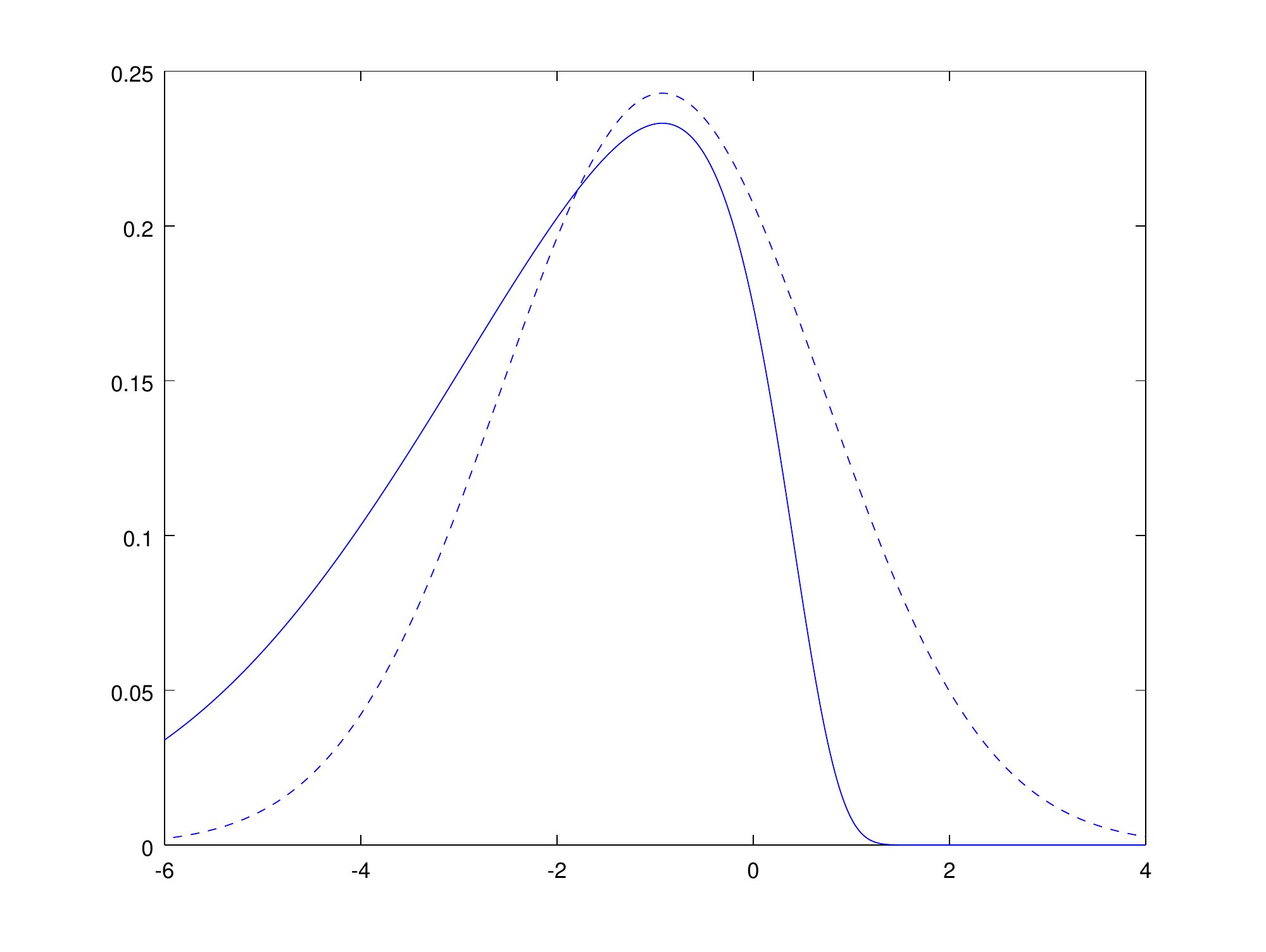}
\includegraphics[width=0.8\textwidth]{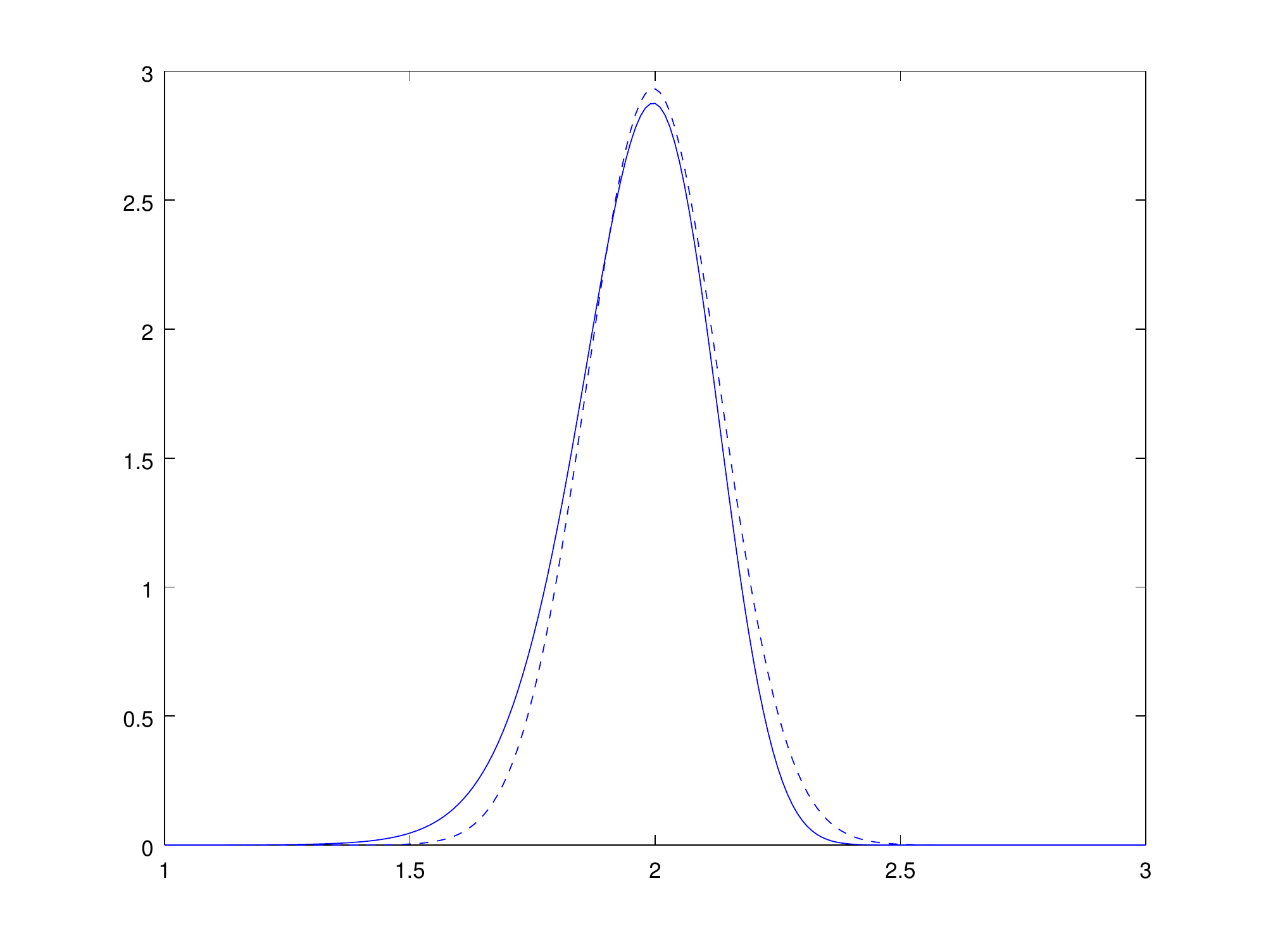}
\caption{Densities of the posterior measure $\mu$ (solid) and its Laplace approximation $\nu$ (dashed) w.r.t. the Lebesgue measure for $y = 2$ (top) and $y=-2$ (bottom)}
\end{figure}

%For each case, we can check the following sequence of conditions (becoming more easily to check but at the same time worsening the quality of the bound on the Hellinger distance).

%\begin{enumerate}
%\item Find optimal values for $A,a$ in lemma \ref{lem:main}. Note that this is only interesting theoretically in order to see how much subsequent steps lose in terms of strictness. As the evaluation of integral \eqref{eq:asslemma} is as difficult as the concrete evaluation of the Hellinger distance, this is not a very useful condition.
%\item Set $a = \frac{1}{A}$ in lemma \ref{lem:main} such that only one constant needs to be found. This has the same drawback as the first step (essentially no reduction in computational complexity of the condition).
%\item Use the inverse inequality machinery from section \ref{sec:reverse} (or more concretely: Find the constant $K$ in theorem \ref{prop:main2} to obtain a bound on the Hellinger distance).
%\item Use the expression in lemma \ref{lem:useful} to obtain a value for $K$.
%\item Use bounds on the term $|\Phi(u) - T\Phi(u)|$ and growth conditions on $\Phi$ and $T\Phi$ to get an upper bound on the value of $K$ in lemma \ref{lem:useful}. This is the only practical condition: While it will be as computationally difficult to do steps 1-4 as it is to compute the Hellinger distance explicitly, error conditions for the Taylor approximation of $\Phi$ are easily acquired.
%\end{enumerate}

\paragraph{The case $y=-2$.}
Recall that the actual value for the Hellinger distance is $d_H(\mu,\nu) \approx 0.32595$
\begin{enumerate}
%\item We can set $A = 1.565$, $a=0.595$ such that lemma \ref{lem:main} is applicable and thus we obtain a theoretical upper bound $d_H(\mu,\nu)\leq 0.32635$.
%\item If we set $a = 1/A$, a valid value for $A$ is $A=1.63$, giving an upper bound $d_H(\mu, \nu) \leq 0.32945$.
\item The $K$ from proposition \ref{prop:main2} is $K\approx 0.46621$, hence $d_H(\mu,\nu)\leq 0.41128$.
\item Explicit calculation of $K^2$ in corollary \ref{cor:useful} yields $K\approx 0.55328$, so $d_H(\mu, \nu)\leq 0.50517$.
%\item \todo[inline]{Falsch!}Take the following (easy to prove) items:
%\begin{enumerate}
%\item $1/4\cdot |\Phi(u)-T\Phi(u)|^2 \leq 0.2$ for $u \in [-4, 0]$.
%\item $-(\Phi(u)\wedge T\Phi(u)) = -\Phi(u) > e^{-e^{-4}/2} > 0.99$ for $u<-4$.
%\item $-(\Phi(u)\wedge T\Phi(u)) = -T\Phi(u)$ for $u >0$.
%\end{enumerate}
%Then with $C = \frac{\sigma \sqrt{I''(\uMAP)}}{\exp(-I(\uMAP)} = (\int e^{-T\Phi}\de\mu_0)^{-1}$, 
%\begin{align*}
%K^2 &= C\cdot \int \exp (-\min\{\Phi(u),T\Phi(u)\})\cdot \min\left\{\frac{|\Phi(u)-T\Phi(u)|^2}{4},1\right\}\mu_0(\de u)\\
%&\leq C\cdot \int_{-\infty}^{-4} 0.99\de\mu_0 + C\cdot \int_{-4}^0 0.2 \de\mu_0 + \frac{\int_0^\infty e^{-T\Phi}\de\mu_0}{\int_{-\infty}^\infty e^{-T\Phi}\de\mu_0}\\
%&\approx 0.36681
%\intertext{and thus}
%K\approx 0.60565
%\intertext{and}
%d_H(\mu,\nu) &\leq 0.56342
%\end{align*}
%Note that the integrals are all simple Gaussian integrals which can be calculated with tabular values of Gaussian quantiles.
\end{enumerate}
 
\paragraph{The case $y=2$.}
Recall that the actual value for the Hellinger distance is $d_H(\mu,\nu) \approx 0.095810$
\begin{enumerate}
%\item We can set $A = 1.16$, $a=0.88$ such that lemma \ref{lem:main} is applicable and thus we obtain a theoretical upper bound $d_H(\mu,\nu)\leq 0.097284$.
%\item If we set $a = 1/A$, a valid value for $A$ is $A=1.15$, giving an upper bound $d_H(\mu, \nu) \leq 0.098427$.
\item The $K$ from proposition \ref{prop:main2} is $K\approx 0.13648$, hence $d_H(\mu,\nu)\leq 0.10330$.
\item Explicit calculation of $K^2$ in corollary \ref{cor:useful} yields $K\approx 0.17422$, so $d_H(\mu, \nu)\leq 0.13434$.
%\item \todo[inline]{Unbrauchbar! Rechnung falsch!}The following facts hold:
%\begin{enumerate}
%\item $1/4\cdot |\Phi(u)-T\Phi(u)|^2 \leq 0.5$ for $u\in[1.6, 2.3]$.
%\item $-(\Phi(u)\wedge T\Phi(u)) = -\Phi(u) > 2.5$ for $u<1.6$.
%\item $-(\Phi(u)\wedge T\Phi(u)) = -T\Phi(u)$ for $u >2.3$.
%\end{enumerate}
%Then with $C = \frac{\sigma \sqrt{I''(\uMAP)}}{\exp(-I(\uMAP)} = (\int e^{-T\Phi}\de\mu_0)^{-1}$, 
%\begin{align*}
%K^2 &=  C\cdot \int \exp (-\min\{\Phi(u),T\Phi(u)\})\cdot \min\left\{\frac{|\Phi(u)-T\Phi(u)|^2}{4},1\right\}\mu_0(\de u)\\
%&\leq C\cdot \int_{-\infty}^{1.6} e^{-2.5}\de\mu_0 + C\cdot \int_{1.6}^{2.3} 0.5 \de\mu_0 + \frac{\int_{2.3}^\infty e^{-T\Phi}\de\mu_0}{\int_{-\infty}^\infty e^{-T\Phi}\de\mu_0}\\
%&\approx 0.36681
%\intertext{and thus}
%K\approx 0.60565
%\intertext{and}
%d_H(\mu,\nu) &\leq 0.56342
%\end{align*}
\end{enumerate} 
In conclusion, the bounds are not strict even in the one-dimensional setting (which is obvious, given the nature of inequalities we have used in lemmata \ref{lem:aux1}, \ref{lem:aux2}, but they give a good idea of the magnitude of the Hellinger distance nonetheless.
%\section{More material}
%In 1d: $\frac{\de\nu}{\de\mu} = \frac{Z_\mu}{Z_\nu} \cdot \frac{\exp(-T\Phi)}{\exp(-\Phi)}$, hence
%\[ D_{KL}(\nu\|\mu) = \int \log \frac{\de\nu}{\de\mu} \de\nu = \log\frac{Z_\mu}{Z_\nu} + \int (\Phi-T\Phi)\de \nu = \log\frac{Z_\mu}{Z_\nu} + \frac{1}{Z_\nu}\int (\Phi-T\Phi)\exp(-T\Phi)\de \mu_0.\]

The calculation of the necessary quantities was done by explicit integration in order to check for strictness of the bound. This is unfeasible in higher dimensions which are of much higher interest (in one dimension there is no reason to even use a Laplace approximation instead of the fully non-Gaussian posterior). So we turn our attention to ways of obtaining results by other methods. We will see that this is difficult even in one dimension: Corollary \ref{cor:useful} demands calculation of the integral 
\[K^2 = \frac{\sqrt{\det(C_0^{1/2}\H I(\uMAP)C_0^{1/2})}}{\exp(-I(\uMAP))}\cdot \int \exp (-\min\{\Phi(u),T\Phi(u)\})\cdot \min\left\{\frac{|\Phi(u)-T\Phi(u)|^2}{4},1\right\}\mu_0(\de u).\]
As a simpler example, consider the integral
\[\int_{-\infty}^\infty x^2 \exp(-|x|) N(0,1)(\de x).\]
An intuitive approach is to split the real line in subsets $A = [-\eps,\eps]$, $B = [-\log(1/\eps), -\eps] \cup [\eps ,\log(1/\eps)]$, $C = \{|u|> \log(1/\eps)\}$. Then the quadratic term is small on $A$ and the exponential term is small on $C$. The set $B$ makes a lot more trouble as neither the quadratic nor the exponential part is consistently small and taking the maximum $x^2 \exp(-|x|)|_{x=2}$ as a bound for the function on $B$ leads to a suboptimal result. It is the same with the formula for $K$: For $u$ near $\uMAP$, the error between $\Phi$ and $T\Phi$ is small, while for large $u$, both the misfit $\Phi$ and the quadratic $T\Phi$ become large, so the exponential term gets small very quickly. The interface between those two domains though is hard to handle and taking an analytical bound ony any term in this interface yields a highly suboptimal bound. This makes an explicit calculation of the integral almost impossible (save taking a large number of subsets tesselating the domain, but this amounts to explicit numerical integration in the limit). This is even more true in higher (and infinite dimensions), so the practical use of this formula is very limited.

This problem makes it necessary to apply MCMC sampling to obtain the value $K$, but in this case we could have used to definition of $d_H$ in the first place to get the exact Hellinger distance. 

\section*{Acknowledgments} The author thanks Serge Kräutle for help with reverse CSB inequalities as well as Carlo Beenakker for a helpful discussion on MathOverflow.

\bibliographystyle{siam}
\bibliography{lit}

\end{document}